\documentclass[11pt]{article}
\usepackage{enumerate}
\usepackage{amssymb,a4wide,latexsym,makeidx,epsfig,fleqn}
\usepackage{amsthm}
\usepackage{amsmath}
\usepackage{enumerate}
\usepackage{extarrows}
\usepackage{graphicx}
\usepackage{subfigure}
\usepackage{float}
\usepackage[justification=centering]{caption}
\newtheorem{theorem}{Theorem}[section]
\newtheorem{remark}[theorem]{Remark}

\newtheorem{definition}[theorem]{Definition}
\newtheorem{lemma}[theorem]{Lemma}

\newtheorem{corollary}[theorem]{Corollary}
\newtheorem{problem}[theorem]{Problem}

\begin{document}
\textwidth 150mm \textheight 225mm
\title{ Merging the A- and Q-spectral theories for digraphs
\thanks{Supported by the National Natural Science Foundation of China (No. 11871398), the Natural Science Basic Research Plan in Shaanxi Province of China (Program No. 2018JM1032) and China Scholarship Council (No. 201706290182).
 }}
\author{{Weige Xi$^1$,  Wasin So$^{2}$, Ligong Wang$^{1,}\footnote{Corresponding author.}$  }\\
{\small $^1$ Department of Applied Mathematics, School of Science,}\\ {\small  Northwestern Polytechnical University, Xi¡¯an, Shaanxi 710072, China.}\\
{\small $^2$ Department of Mathematics and Statistics, San Jose State University, }\\
{\small San Jose, CA 95192-0103, USA.}\\
{\small E-mail:  xiyanxwg@163.com, lgwangmath@163.com, wasin.so@sjsu.edu}\\}

\date{}
\maketitle
\begin{center}
\begin{minipage}{120mm}
\vskip 0.3cm
\begin{center}
{\small {\bf Abstract}}
\end{center}
{\small  \ Let $G$ be a digraph and $A(G)$
be the adjacency matrix of $G$. Let $D(G)$ be the diagonal matrix with outdegrees of
vertices of $G$. For any real $\alpha\in[0,1]$, Liu et al. \cite{LWCL}
defined the matrix $A_\alpha(G)$ as
$$A_\alpha(G)=\alpha D(G)+(1-\alpha)A(G).$$
The largest modulus of the eigenvalues
of $A_\alpha(G)$ is called the $A_\alpha$ spectral radius of $G$. In this paper, we
determine the digraphs which attain the maximum (or minimum)
$A_\alpha$ spectral radius among all strongly connected digraphs with given parameters such as girth, clique number,
vertex connectivity or arc connectivity. We also discuss a number of open problems.

}
\end{minipage}
\end{center}

\section{Introduction }

Let
$G=(V(G),E(G))$
be a digraph with vertex set
$V(G)=\{v_1,v_2,\ldots,v_n\}$ and arc set
$E(G)$. A digraph is simple if it has no loops and multiple arcs. A digraph is strongly connected if for every pair of vertices $v_i,v_j\in
V(G)$, there exists a directed path from $v_i$ to
$v_j$. Throughout this paper, we only
consider simple strongly connected digraphs.

Let $\overset{\longrightarrow}{P_{n}}$ and
$\overset{\longrightarrow}{C_{n}}$ denote the directed path and the
directed cycle on $n$ vertices, respectively. Let $\overset{\longleftrightarrow}{K_{n}}$
denote the complete digraph on $n$ vertices in which for two arbitrary distinct vertices
$v_i,v_j\in V(\overset{\longleftrightarrow}{K_{n}})$, there are arcs
$(v_i,v_j)$ and $(v_j,v_i)\in E(\overset{\longleftrightarrow}{K_{n}})$. Suppose
$\overset{\longrightarrow}{P_k}=v_1v_2\dots v_k$, we call $v_1$ the
initial vertex of the directed path
$\overset{\longrightarrow}{P_k}$, and $v_k$ the terminal vertex of the
directed path $\overset{\longrightarrow}{P_k}$.

Let $G$ be a digraph. If $S\subset V(G)$, then we use $G[S]$ to denote the subdigraph of $G$ induced by $S$.
Let $G-v$ be a digraph obtained from $G$ by deleting the vertex $v$ and all arcs incident to $v$. We use $G\pm e$ to denote
the digraph obtained from $G$ by adding/deleting the arc $e\notin E(G)$/$e\in E(G)$. Let $G_1$ and $G_2$ be two disjoint digraphs.
The digraph $G_1\cup G_2$ is the digraph with vertex set $V(G_1)\cup V(G_2)$ and arc set $E(G_1)\cup E(G_2)$.
We denote by $G_1\vee G_2$ the join of $G_1$ and $G_2$, which is the digraph such that $V(G_1\vee G_2) =V(G_1)\cup V(G_2)$
and $E(G_1\vee G_2) =E(G_1)\cup E(G_2)\cup\{(u,v),(v,u): u\in V(G_1) \ \textrm{and} \ v\in V(G_2)\}$.

Let $H$ be a subdigraph of $G$.
If $G[V(H)]$ is a complete subdigraph of $G$, then $H$ is called a clique of $G$. The clique number of a digraph $G$,
denoted by $\omega(G)$, is the maximum value of the numbers of the vertices of the cliques in $G$.
The girth of $G$ is the length of the shortest directed cycle of $G$. For a strongly connected digraph $G=(V(G),E(G))$,
the vertex connectivity of $G$, denoted by $\kappa(G)$, is the minimum
number of vertices whose deletion yields the resulting digraph non-strongly connected. A
set of arcs $S\subset E(G)$ is an arc cut set if $G-S$ is not strongly connected.
The arc connectivity of $G$, denoted by $\kappa'(G)$, is the minimum
number of arcs whose deletion yields the resulting digraph not-strongly connected.

For a digraph $G$, if there is an arc from $v_i$ to $v_j$, we indicate this by writing
$(v_i,v_j)$, call $v_j$ the head of $(v_i,v_j)$, and $v_i$ the tail
of $(v_i,v_j)$, respectively, and $(v_i,v_j)$ is said to be out-incident to $v_i$ and in-incident to $v_j$;
$v_i$ is said to be out-adjacent to $v_j$ and $v_j$ is said to be in-adjacent to $v_i$. A tournament is a directed graph
obtained by assigning a direction for each edge in an undirected complete graph. A transitive
tournament is a tournament $G$ satisfying the following: if $(u, v)\in E(G)$ and
$(v, w)\in E(G)$, then $(u, w)\in E(G)$.

For any vertex $v_i$,
let $N_i^{+}=N^+_{v_i}=\{v_j\in V(G) \mid (v_i,v_j)\in
E(G)\}$ and $N_i^{-}=N^-_{v_i}=\{v_j\in V(G)
\mid (v_j,v_i)\in E(G)\}$ denote the out-neighbors
and in-neighbors of $v_i$, respectively. Let $d_i^{+}=d_{v_i}^+=|N_i^{+}|$
denote the outdegree of the vertex $v_i$, and $d_i^{-}=d_{v_i}^-=|N_i^{-}|$
denote the indegree of the vertex $v_i$ in the digraph $G$. The minimum outdegree is denoted by $\delta^+$,
the maximum outdegree is denoted by $\Delta^+$ and the minimum indegree by $\delta^-$. A digraph
is $r$-regular if all vertices have outdegree $r$ and indegree $r$.

For a digraph $G$, let
$A(G)=(a_{ij})_{n\times n}$ be the adjacency matrix
of $G$, where $a_{ij}=1$ if $(v_i,v_j)\in
E(G)$ and $a_{ij}=0$ otherwise. Let
$D(G)$ be the diagonal matrix with outdegrees of
vertices of a digraph $G$. In this paper we study hybrids of $A(G)$ and $D(G)$
similar to the signless Laplacian matrix $Q(G)=D(G)+A(G)$, which has been
 extensively studied since then. For detailed coverage
of this research see \cite{HoYo,LaWa,LWZ,LWM,XiWa1,XiWa2}, and their references.
The study of $Q(G)$ has
shown that it is a remarkable matrix, and unique in many respects. Yet, $Q(G)$ is just
the sum of $A(G)$ and $D(G)$. To understand to what extent each of the summands
$A(G)$ and $D(G)$ determines the properties of $Q(G)$, Liu et al. \cite{LWCL}
defined the matrix $A_\alpha(G)$ as

$$A_\alpha(G)=\alpha D(G)+(1-\alpha)A(G), \ \ 0\leq \alpha\leq1.$$
Many facts suggest that the study of the family $A_\alpha(G)$ is long due. Our inspiration comes from the paper of Nikiforov \cite{Niki}. First, we note that
$$A(G)=A_0(G), \ \ \  D(G)=A_1(G), \ \ \ \textrm{and} \ \ \ Q(G)=2A_{\frac{1}{2}}(G).$$
Since $A_{\frac{1}{2}}(G)$ is essentially equivalent to $Q(G)$, in this paper
we take $A_{\frac{1}{2}}(G)$ as
an exact substitute for $Q(G)$. With this caveat, one sees that $A_\alpha(G)$ seamlessly
joins $A(G)$ with $Q(G)$, and we may study the adjacency spectral
properties and signless Laplacian spectral properties of a digraph in a unified way.
The spectral radius of
$A_\alpha(G)$ i.e., the largest modulus of the eigenvalues
of $A_\alpha(G)$, is called the $A_\alpha$ spectral radius of $G$, denoted by
$\lambda_\alpha(G)$. The $A_\alpha$ spectral radius of  undirected graphs has been studied in the literature, see \cite{GZ,LHX,LCM,Niki,NPRS,NR,XLLS}.
Recently, Liu et al. \cite{LWCL} characterized
the extremal digraph which attains the maximum $A_\alpha$ spectral radius among all strongly connected
digraphs with given dichromatic number. We are interested in the $A_\alpha$ spectral radius of digraphs with given other parameters.

If $\alpha=1$, $A_1(G)=D(G)$ the diagonal matrix with outdegrees of
vertices of $G$ which is not interesting. So we only consider the cases $0\leq \alpha<1$ for the rest of this paper. If $G$ is a
strongly connected digraph, then it follows from the Perron Frobenius Theorem \cite{HJ} that $\lambda_\alpha(G)$ is an eigenvalue of $A_\alpha(G)$,
and there is a unique positive unit eigenvector corresponding to $\lambda_\alpha(G)$. The positive unit eigenvector
corresponding to $\lambda_\alpha(G)$ is called the Perron vector of $A_\alpha(G)$. See more details on the Perron vector of $A_\alpha(G)$ in Section 2.

One of the central issues in extremal spectra graph theory is: for a graph matrix, determine the maximum
or minimum spectral radius over various families of graphs. For example, among all strongly connected digraphs on $n$ vertices,
$\overset{\longrightarrow}{C_{n}}$ is the unique digraph with the minimum $A_0$ spectral radius and $A_{\frac{1}{2}}$ spectral radius, and
$\overset{\longleftrightarrow}{K_{n}}$ is the unique digraph with the maximum $A_0$ spectral radius and $A_{\frac{1}{2}}$ spectral radius. The same result
is also true for the $A_\alpha$ spectral radius of $G$, see Corollary \ref{co:c2}. The main goal of this paper is to extend some results
on maximum or minimum $A_0$ spectral radius and $A_{\frac{1}{2}}$ spectral radius for all $\alpha\in[0,1)$.

The rest of the paper is structured as follows. In the next section we introduce
some lemmas and give basic facts about the $A_\alpha$ spectral radius of $G$. In Section 3,
we characterize the extremal digraph which achieves the minimum $A_\alpha$ spectral radius among all strongly connected
digraphs with given girth. In Section 4, we determine the extremal digraph which attains the minimum $A_\alpha$ spectral radius among all strongly connected
digraphs with given clique number. In Section 5, we characterize the extremal digraphs
which achieve the maximum $A_\alpha$ spectral radius among all strongly connected digraphs with given
vertex connectivity. In Section 6, we characterize the extremal digraphs
which achieve the maximum $A_\alpha$ spectral radius among all strongly connected digraphs with given
arc connectivity.

\section{Preliminaries }

 In this section, we give some lemmas which will be used in the following sections.

Let $\sigma(\cdot)$ denote
the spectrum of a square matrix including algebraic multiplicity. Let $\rho(\cdot)$ denote the
spectral radius of a square matrix.

\noindent\begin{lemma}\label{le:1} (\cite{HJ}) \ Let $M=(m_{ij})$
be an $n \times n$ nonnegative matrix, $R_i(M)$ be the $i$-th row sum of $M$, i.e.,
$R_i(M)=\sum\limits_{j=1}^n m_{ij} \ (1 \leq i\leq n)$. Then
$$\min\{R_{i}(M):1 \leq i\leq n\}\leq \rho(M)\leq \max\{R_{i}(M):1
\leq i\leq n\}.$$
 Moreover, if $M$ is irreducible, then either one equality holds  if
and only if $R_1(M)=R_2(M)=\ldots=R_n(M)$.
\end{lemma}

\begin{definition}\label{de:2} (\cite{LWe}) Let $M$ be a real matrix of order $n$ described in the following block
form
$$
M=\left(
  \begin{array}{cccc}
     M_{11} & M_{12} & \cdots & M_{1t}\\
     M_{21} &M_{22} & \cdots & M_{2t}\\
    \vdots  & \vdots   & \ddots & \vdots\\
     M_{t1} & M_{t2} & \cdots & M_{tt}\\
  \end{array}
\right),
$$
where the diagonal blocks $M_{ii}$ are $a_i\times a_i$ matrices for any $i\in\{1, 2,\cdots, t\}$ and $n=a_1+\cdots+a_t$.
For any $i, j\in\{1, 2,\cdots, t\}$, let $b_{ij}$ denote the average row sum of $M_{ij}$ , i.e. $b_{ij}$ is the sum of
all entries in $M_{ij}$ divided by the number of rows. Then $B(M) = (b_{ij})$ is called
the quotient matrix of $M$. If in addition for each pair $i, j$, $M_{ij}$ has constant row sum, then
$B(M)$ is called the equitable quotient matrix of $M$.
\end{definition}

\begin{lemma}\label{le:c1} (\cite{YYWX}) Let $M=(m_{ij})_{n\times n}$ be defined as above, and for any $i, j\in \{1, 2,\ldots, t\}$, the row
sum of each block $M_{ij}$ be constant. Let $B= (b_{ij})$ be the equitable quotient matrix of $M$. Then $\sigma(B)\subset\sigma(M)$.
Moreover, if $M=(m_{ij})_{n\times n}$ is a nonnegative matrix, then $\rho(B)=\rho(M)$.
\end{lemma}

\begin{definition}\label{de:1} (\cite{BP,HJ}) Let $A = (a_{ij})$ and $B = (b_{ij})$ be $n\times m$ matrices.
If $a_{ij} \leq b_{ij}$ for all $i$ and $j$, then $A \leq B$. If $ A \leq B$ and $A\neq B$, then $A < B$.
If $a_{ij}< b_{ij}$ for all $i$ and $j$, then $A\ll B$.
\end{definition}

\begin{lemma}\label{le:2} (\cite{BP,HJ}) Let $A$ and $B$ be nonnegative matrices. If $0 \leq A \leq B$,
then $\rho(A) \leq \rho(B)$. Furthermore, if $B$ is irreducible and $0\leq A< B$, then $\rho(A)<\rho(B)$.
\end{lemma}

By Lemma \ref{le:2}, we have the following results in terms of $A_\alpha$ spectral radius of digraphs.

\begin{corollary}\label{co:1} Let $G$ be a digraph and $H$ be a spanning subdigraph of $G$. Then

(i) $\lambda_\alpha(G)\geq \lambda_\alpha(H)$.

(ii) If $G$ is strongly connected, and $H$ is a proper subdigraph of $G$,
then $\lambda_\alpha(G)>\lambda_\alpha(H)$.
\end{corollary}

From Lemma \ref{le:1} and Corollary \ref{co:1}, we can easily get the following corollary.

\begin{corollary}\label{co:c2} Let $G$ be a strongly connected digraph. Then
$1\leq\lambda_\alpha(G)\leq n-1$, $\lambda_\alpha(G)=n-1$ if and only if $G\cong\overset{\longleftrightarrow}{K_{n}}$, and
$\lambda_\alpha(G)=1$ if and only if $G\cong\overset{\longrightarrow}{C_{n}}$.
\end{corollary}

\noindent\begin{lemma}\label{le:3} (\cite{BP}) Let $B$ be nonnegative matrices and
$X=(x_1,x_2,$ $\ldots,x_n)^T$ be any nonzero nonnegative vector. If $\beta\geq0$ such that $BX\geq\beta X$, then $\rho(B)\geq\beta$. Furthermore,
if $B$ is irreducible and $BX>\beta X$, then $\rho(B)>\beta$.
\end{lemma}

By Lemma \ref{le:3}, we have the following results in terms of $A_\alpha$ spectral radius of digraphs.

\begin{corollary}\label{co:c3} Let $G$ be a strongly connected digraph. Then
$\lambda_\alpha(G)> \alpha\Delta^+$.
\end{corollary}

\begin{proof} Without loss of generality, let $d^+_u=\Delta^+$. Taking $X=(0,0,\ldots,0,1,0\ldots,0)^T$,
that is, all the entries of $X$ are 0 except $x_u=1$, where $x_u$ corresponding to the vertex $u$. Since
$G$ is strongly connected, then $d^-_u\geq1$ and $A_\alpha(G)$ is nonnegative irreducible. Hence $A_\alpha(G)X>\alpha\Delta^+ X$.
Therefore, by Lemma \ref{le:3}, we have $\lambda_\alpha(G)> \alpha\Delta^+$.
\end{proof}

In the rest of this section, let $X=(x_1,x_2,\ldots,x_n)^T$ be the unique positive unit eigenvector
corresponding to the $A_\alpha$ spectral radius
$\lambda_\alpha(G)$, where $x_i$ corresponds to the vertex $v_i$.

\noindent\begin{lemma}\label{le:4} Let $G=(V(G),E(G))$
be a strongly connected digraph on $n$ vertices, $v_p, v_q$ be two distinct vertices of $V(G)$. Suppose that $v_1,v_2,\ldots,v_t\in N^-_{v_p}\setminus\{N^-_{v_q}\cup\{v_q\}\}$, where $1\leq t\leq d_p^-$.
Let $H=G-\{(v_i,v_p): i=1,2\ldots,t\}+\{(v_i,v_q): i=1,2\ldots,t\}$. If $x_{v_q}\geq x_{v_p}$, then $\lambda_\alpha(H)\geq \lambda_\alpha(G)$.
Furthermore, if $H$ is strongly connected and $x_{v_q}>x_{v_p}$, then $\lambda_\alpha(H)>
\lambda_\alpha(G)$.
\end{lemma}

\begin{proof} We will show $(A_\alpha(H)X)_i\geq(A_\alpha(G)X)_i$ for any $v_i\in V(G)$.

{\bf Case 1.} $v_i\notin N^-_{v_p}\setminus\{N^-_{v_q}\cup\{v_p\}\}$.

Then $(A_\alpha(H)X)_i=\alpha d_i^+(G)x_i +(1-\alpha)\sum\limits_{(v_i,v_j)\in E(G)}x_j=(A_\alpha(G)X)_i$.

{\bf Case 2.} $v_i\in N^-_{v_p}\setminus\{N^-_{v_q}\cup\{v_p\}\}$.

Then $(A_\alpha(H)X)_i-(A_\alpha(G)X)_i=(1-\alpha)(x_{v_q}-x_{v_p})\geq0$.

Thus $A_\alpha(H)X\geq A_\alpha(G)X=\lambda_\alpha(G)X$. By Lemma \ref{le:3}, $\lambda_\alpha(H)\geq
\lambda_\alpha(G)$.

Moreover, if $H$ is strongly connected and $x_{v_q}>x_{v_p}$, then by Lemma \ref{le:3}, we have $\lambda_\alpha(H)>
\lambda_\alpha(G)$.
\end{proof}

\noindent\begin{lemma}\label{le:5} Let $G$ ($\neq\overset{\longrightarrow}{C_{n}}$) be a strongly connected digraph
with $V(G)=\{v_1,v_2,\cdots,v_n\}$. Suppose that $\overset{\longrightarrow}{P_k}=v_1v_2\dots v_k$ $(k\geq3)$ be a directed path of
$G$ with $d_i^+=1$ for $(i=2,3,\ldots,k-1)$. Then we have $x_2<x_3<\ldots<x_{k-1}<x_k$.
\end{lemma}

\begin{proof} Since $G$ is a strongly connected digraph and $G\neq\overset{\longrightarrow}{C_{n}}$,
then $D$ contains a directed cycle, denoted by $\overset{\longrightarrow}{C_{g}}$ ($g\geq2$), as a proper subdigraph of $G$.
Thus $\lambda_\alpha(G)>\lambda_\alpha(\overset{\longrightarrow}{C_{g}})=1$ by Corollary \ref{co:1}. Therefore, for any $i\in\{2,3,\ldots,k-1\}$,
we have  $$x_i<\lambda_\alpha(G)x_i=\alpha x_i+(1-\alpha)x_{i+1}.$$ Then $x_i<x_{i+1}$ and thus $x_2<x_3<\ldots<x_{k-1}<x_k$.
\end{proof}
\noindent\begin{lemma}\label{le:6} (\cite{BP}) Let $B$ be nonnegative matrices, then
$B$ is reducible if and only if $\rho(B)$ is the spectral radius of some proper principal submatrix of $B$.
\end{lemma}
\noindent\begin{lemma}\label{le:7} Let $G$ ($\neq\overset{\longrightarrow}{C_{n}}$) be a strongly connected digraph with $V(G)=\{v_1,v_2,\cdots,v_n\}$, $(v_i,v_j)\in E(G)$ and $w\notin V(G)$, $G^w=(V(G^w),E(G^w))$ with $V(G^w)=V(G)\cup\{w\}$,
$E(G^w)=E(G)-\{(v_i,v_j)\}+\{(v_i,w),(w,v_j)\}$. Then
$\lambda_\alpha(G)\geq\lambda_\alpha(G^w)$.
\end{lemma}

\begin{proof} Since $G$ is strongly connected digraph, then $G^w$ is also strongly connected digraph.
Suppose ${\bf X}=(x_1,\ldots,x_n,x_w)^T$ is the Perron vector corresponding
to $\lambda_\alpha(G^w)$, where $x_w$ correspond to $w$, and $x_i$ correspond to $v_i$ for $i=1,2,\cdots,n$. Since
$G\neq\overset{\longrightarrow}{C_{n}}$, then $G^w\neq\overset{\longrightarrow}{C_{n}}_{+1}$. By Lemma \ref{le:1}, we
have $\lambda_\alpha(G^w)>1$. Clearly, $d_w^+(G^w)=1$. Thus $x_w<\lambda_\alpha(\theta(G^w)x_w=\alpha x_w+(1-\alpha)x_j$. Then $x_j>x_w$. Let $H=G^w-\{(v_i,w)\}+\{(v_i,v_j)\}$, then by
$d_w^-(G^w)=1$, $H$ is not strongly connected which has exactly two strongly connected components, one is isolated vertex $w$, the other is $G$. Thus
$A_\alpha(H)$ is nonnegative reducible, then by Lemma \ref{le:6}, we have $\lambda_\alpha(H)=\lambda_\alpha(G)$. On the other hand, by Lemma \ref{le:4}, $\lambda_\alpha(H)\geq\lambda_\alpha(G^w)>1$.
Thus $\lambda_\alpha(G)=\lambda_\alpha(H)\geq\lambda_\alpha(G^w)$.
\end{proof}

\section{The minimum $A_\alpha$ spectral radius of strongly connected digraphs with given girth}

Let $g\geq2$ and $\mathcal{G}_{n,g}$ denote the set of strongly connected digraph on $n$ vertices with girth $g$. If $g=n$, then
$\mathcal{G}_{n,g}=\{\overset{\longrightarrow}{C_{n}}\}$ and $\lambda_\alpha(\overset{\longrightarrow}{C_{n}})=1$. Thus we only need to consider
the cases $2\leq g\leq n-1$.

Let $2\leq g\leq n-1$ and $C_{n,g}$ be a digraph obtained by adding a directed path $\overset{\longrightarrow}{P_{n}}_{-g+2}=v_{g}v_{g+1}\dots v_nv_1$
on the directed cycle $\overset{\longrightarrow}{C_{g}}=v_{1}v_{2}\dots v_gv_1$ such that
$V(\overset{\longrightarrow}{C_{g}})\cap V(\overset{\longrightarrow}{P_{n}}_{-g+2})=\{v_g,v_1\}$ (as shown in Figure \ref{Fig.2.}), where $V(C_{n,g})=\{v_1,v_2,\ldots,v_n\}$.
Clearly, $C_{n,g}\in\mathcal{G}_{n,g}$.
\begin{figure}[H]
\begin{centering}
\includegraphics[scale=0.8]{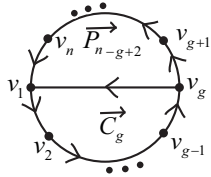}
\caption{The digraph $C_{n,g}$}\label{Fig.2.}
\end{centering}
\end{figure}

In \cite{LSWY}, Lin and Shu et al. proved that $C_{n,g}$ attains the minimum $A_0$ spectral radius among all strongly connected
digraphs with given girth. In \cite{HoYo}, Hong and You proved that $C_{n,g}$ also attains the minimum $A_{\frac{1}{2}}$ spectral radius
 among all strongly connected digraphs with given girth. We generalize their results to $0\leq\alpha<1$. In the rest of this section, we will show that $C_{n,g}$ achieves the
minimum $A_\alpha$ spectral radius among all digraphs in $\mathcal{G}_{n,g}$.

\noindent\begin{lemma}\label{le:8} Let $2\leq g\leq n-1$ and $C_{n,g}'=C_{n,g}-\{(v_n,v_1)\}+\{(v_n,v_g)\}$. Then
$\lambda_\alpha(C_{n,g}')>\lambda_\alpha(C_{n,g})$.
\end{lemma}

\begin{proof} Suppose ${\bf X}=(x_1,x_2\ldots,x_n)^T$ is the Perron vector corresponding
to $\lambda_\alpha(C_{n,g})$, where $x_i$ correspond to $v_i$ for $i=1,2,\cdots,n$. Since $\overset{\longrightarrow}{D}=v_{g+1}v_{g+2}\ldots v_nv_1\ldots v_g$
and $\overset{\longrightarrow}{R}=v_gv_{g+1}v_{g+2}$ are the directed paths of $C_{n,g}$ with $d_{v_i}^+=1$ for $i\in\{1,2,,\ldots,g-1,g+1,g+2,\ldots,n\}$,
then by Lemma \ref{le:5}, we have $x_{g+2}<x_{g+3}<\ldots<x_n<x_1<x_2<\ldots<x_{g-1}<x_g$ and $x_{g+1}<x_{g+2}$. Thus
$x_{g+1}<x_{g+2}<x_{g+3}<\ldots<x_n<x_1<x_2<\ldots<x_{g-1}<x_g$.

Note that $C_{n,g}'$ is strongly connected and $x_g>x_1$, then by Lemma \ref{le:4}, we have $\lambda_\alpha(C_{n,g}')$ $>\lambda_\alpha(C_{n,g})$.
\end{proof}

\noindent\begin{theorem}\label{th:c-1}  Let $2\leq g\leq n-1$ and $G\in\mathcal{G}_{n,g}$. Then
$\lambda_\alpha(G)\geq\lambda_\alpha(C_{n,g})>1$ with equality if and only if $G\cong C_{n,g}$.
\end{theorem}
\begin{proof} Since $\overset{\longrightarrow}{C_{g}}$ is a proper subdigraph of $G$,
$\lambda_\alpha(G)>\lambda_\alpha(\overset{\longrightarrow}{C_{g}})=1$ by Corollary \ref{co:1}. Without loss of generality, we let
$\overset{\longrightarrow}{C_{g}}=v_{1}v_{2}\dots v_gv_1$, where $2\leq g\leq n-1$. Since $G\in\mathcal{G}_{n,g}$ is strongly connected,
it is possible to obtain a digraph $G_1$ from $G$ by deleting vertices and arcs in such a way that $G_1\cong H$,
where $H=(V(H),E(H))$, $V(H)=\{v_1,v_2,\ldots,v_g,v_{g+1},\ldots,v_{g+l-2}\}$, $E(H)=\{(v_i,v_{i+1})| i\in\{1,2,\ldots,g+l-3\}\}\cup\{(v_g,v_1),(v_{g+l-2},v_t)\}$
with $1\leq t\leq g$ (see Figure \ref{Fig.3.}).
\begin{figure}[H]
\begin{centering}
\includegraphics[scale=0.8]{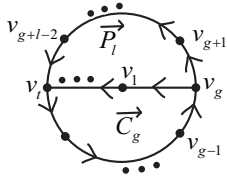}
\caption{The digraph $H$}\label{Fig.3.}
\end{centering}
\end{figure}

By Corollary \ref{co:1}, we have $\lambda_\alpha(G)\geq\lambda_\alpha(G_1)$ with equality if and only if $G\cong G_1$.
Since $G_1$ contains a directed path $\overset{\longrightarrow}{P_{l}}=v_{g}v_{g+1}\dots v_{g+l-2}v_t$. Insert $n-g-l+2$
vertices to $\overset{\longrightarrow}{P_{l}}$ such that the resulting digraph is denoted by $H'$. Clearly,
$H'$ is strongly connected. By using Lemma \ref{le:7} by $n-g-l+2$ times, we have $\lambda_\alpha(H)\geq\lambda_\alpha(H')$ with equality
if and only if $H\cong H'$. Then we consider the following three cases.

{\bf Case 1:} $t=1$.

In this case, $H'\cong C_{n,g}$, then $\lambda_\alpha(G)\geq\lambda_\alpha(G_1)=\lambda_\alpha(H)\geq\lambda_\alpha(H')=\lambda_\alpha(C_{n,g})$,
with equality if and only if $G\cong C_{n,g}$.

{\bf Case 2:} $t=g$.

In this case,$H'\cong C_{n,g}'$, then $\lambda_\alpha(G)\geq\lambda_\alpha(G_1)=\lambda_\alpha(H)
\geq\lambda_\alpha(H')=\lambda_\alpha(C_{n,g}')>\lambda_\alpha(C_{n,g})$ by Lemma \ref{le:8}.

{\bf Case 3:} $2\leq t\leq g-1$.

${\bf X}=(x_1,x_2\ldots,x_n)^T$ is the Perron vector corresponding
to $\lambda_\alpha(C_{n,g})$, where $x_i$ correspond to $v_i$ for $i=1,2,\cdots,n$. Noting that
$H'\cong C_{n,g}-\{(v_n,v_1)\}+\{(v_n,v_t)\}$, by the proof of Lemma \ref{le:8}, we have $x_1<x_t$,
then $\lambda_\alpha(H')>\lambda_\alpha(C_{n,g})$ by Lemma \ref{le:4}. Thus $\lambda_\alpha(G)\geq\lambda_\alpha(G_1)=\lambda_\alpha(H)
\geq\lambda_\alpha(H')>\lambda_\alpha(C_{n,g})$.

In all cases, $\lambda_\alpha(G)\geq\lambda_\alpha(C_{n,g})>1$ with equality if and only if $G\cong C_{n,g}$.
\end{proof}

\section{The minimum $A_\alpha$ spectral radius of strongly connected digraphs with given clique number}

Let $\mathcal{C}_{n,d}$ denote the set of strongly connected digraphs on $n$ vertices with clique number $d$. If $d=n$, then
$\mathcal{C}_{n,d}=\{\overset{\longleftrightarrow}{K_{n}}\}$ and $\lambda_\alpha(\overset{\longleftrightarrow}{K_{n}})=n-1$. If $d=1$, then
$\overset{\longrightarrow}{C_{n}}\in\mathcal{C}_{n,d}$ and $\lambda_\alpha(\overset{\longrightarrow}{C_{n}})=1$. By Corollary \ref{co:c2}, for
any $G\in\mathcal{C}_{n,d}$, $\lambda_\alpha(G)\geq1=\lambda_\alpha(\overset{\longrightarrow}{C_{n}})$
with equality if and only if $G\cong\overset{\longrightarrow}{C_{n}}$. Thus we only need to consider
the cases $2\leq d\leq n-1$.

Let $2\leq d\leq n-1$, and $B_{n,d}$ be a digraph obtained by adding a directed path $\overset{\longrightarrow}{P_{n}}_{-d+2}=v_1v_2\ldots v_{n-d+2}$
to a clique $\overset{\longleftrightarrow}{K_{d}}$ such that
$V(\overset{\longleftrightarrow}{K_{d}})\cap V(\overset{\longrightarrow}{P_{n}}_{-d+2})=\{v_{n-d+2},v_1\}$ (as shown in Figure \ref{Fig.4.}), where $V(B_{n,d})=\{v_1,v_2,\ldots,v_n\}$.
Clearly, $B_{n,d}\in\mathcal{C}_{n,d}$.
\begin{figure}[H]
\begin{centering}
\includegraphics[scale=0.8]{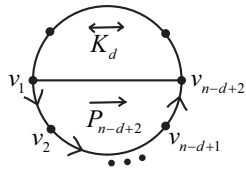}
\caption{The digraph $B_{n,d}$}\label{Fig.4.}
\end{centering}
\end{figure}

In \cite{LSWY}, Lin and Shu et al. proved that $B_{n,d}$ attains the minimum $A_0$ spectral radius among all strongly connected
digraphs with given clique number. In \cite{HoYo}, Hong and You determined that $B_{n,d}$ also attains the minimum $A_{\frac{1}{2}}$ spectral radius
 among all strongly connected digraphs with given clique number. We generalize their results to $0\leq\alpha<1$. In the rest of this section,
 we will show that $B_{n,d}$ achieves the
minimum $A_\alpha$ spectral radius among all digraphs in $\mathcal{C}_{n,d}$.

\noindent\begin{lemma}\label{le:9} Let $2\leq d\leq n-1$ and $B_{n,d}'=B_{n,d}-\{(v_{n-d+1},v_{n-d+2})\}+\{(v_{n-d+1},v_1)\}$. Then
$\lambda_\alpha(B_{n,d}')>\lambda_\alpha(B_{n,d})$.
\end{lemma}

\begin{proof} Suppose ${\bf X}=(x_1,x_2\ldots,x_n)^T$ is the Perron vector corresponding
to $\lambda_\alpha(B_{n,d})$, where $x_i$ correspond to $v_i$ for $i=1,2,\cdots,n$.
By Lemma \ref{le:4}, we only need to show that $x_{n-d+2}<x_1$.

Since
$\overset{\longleftrightarrow}{K_{d}}$ is a proper sundigraph of $\lambda_\alpha(B_{n,d})$, then $\lambda_\alpha(B_{n,d})>d-1$ by Corollary \ref{co:1}.
Therefore, from $A_\alpha(B_{n,d}){\bf X}=\lambda_\alpha(B_{n,d}){\bf X}$, we have
$$\lambda_\alpha(B_{n,d})x_1=\alpha dx_1+(1-a)x_2+(1-\alpha)x_{n-d+2}+(1-\alpha)\sum\limits_{v_i\in V_1}x_i,$$
$$\lambda_\alpha(B_{n,d})x_{n-d+2}=\alpha (d-1)x_{n-d+2}+(1-a)x_1+(1-\alpha)\sum\limits_{v_i\in V_1}x_i,$$
where $V_1=V(\overset{\longleftrightarrow}{K_{d}})\setminus\{v_{n-d+2},v_1\}$.

Then $(\lambda_\alpha(B_{n,d})-\alpha d+1-\alpha)(x_1-x_{n-d+2})=\alpha x_{n-d+2}+(1-\alpha)x_2>0$. By Corollary \ref{co:c3},
$\lambda_\alpha(B_{n,d})>\alpha d$. Thus $x_1>x_{n-d+2}$.
\end{proof}

\noindent\begin{theorem}\label{th:c-2}  Let $2\leq d\leq n-1$ and $G\in\mathcal{C}_{n,d}$. Then
$\lambda_\alpha(G)\geq\lambda_\alpha(B_{n,d})$ with equality if and only if $G\cong B_{n,d}$.
\end{theorem}
\begin{proof} Since $G\in\mathcal{C}_{n,d}$, $\overset{\longleftrightarrow}{K_{d}}$ is a proper subdigraph of $G$.
Since $G$ is strongly connected, it is possible to obtain a digraph $G_1$ from $G$ by deleting vertices and arcs in such a way that either
$G_1\cong B_{d+l-2,d}$ ($l\geq3$) or $G_1\cong B_{d+l-1,d}'$ ($l\geq2$). By Corollary \ref{co:1}, we have $\lambda_\alpha(G)\geq\lambda_\alpha(G_1)$
with equality if and only if $G\cong G_1$. Then we consider the following two cases.

{\bf Case 1:} $G_1\cong B_{d+l-2,d}$ ($l\geq3$).

Insert $n-d-l+2$
vertices to $\overset{\longrightarrow}{P_{l}}$ such that the resulting digraph is $B_{n,d}$.
By using Lemma \ref{le:7} by $n-d-l+2$ times, we have $\lambda_\alpha(G_1)\geq\lambda_\alpha(B_{n,d})$ with equality
if and only if $G_1\cong B_{n,d}$.

{\bf Case 2:}  $G_1\cong B_{d+l-2,d}'$ ($l\geq2$).

Insert $n-d-l+2$
vertices to $\overset{\longrightarrow}{P_{l}}$ such that the resulting digraph is $B_{n,d}'$.
By using Lemma \ref{le:7} by $n-d-l+2$ times, we have $\lambda_\alpha(G_1)\geq\lambda_\alpha(B_{n,d})'$ with equality
if and only if $G_1\cong B_{n,d}'$. But by Lemma \ref{le:9}, $\lambda_\alpha(B_{n,d}')>\lambda_\alpha(B_{n,d})$.
Thus $\lambda_\alpha(G)\geq\lambda_\alpha(G_1)\geq\lambda_\alpha(B_{n,d}')>\lambda_\alpha(B_{n,d})$.

Combining the above arguments, $\lambda_\alpha(G)\geq\lambda_\alpha(B_{n,d})$ with equality if and only if $G\cong B_{n,d}$.
\end{proof}

A tournament on $n$ vertices that maximizes the spectral radius of its $A_\alpha$ matrix among all such tournaments is called extremal.
For $\alpha=0$, it has long been known that if $n$ is odd, the extremal tournaments are precisely the
ones that are regular, i.e. have indegree and outdegree $(n-1)/2$ at each vertex. For $n$ even, the extremal
tournaments are those which are isomorphic to the Brualdi-Li tournament. This was conjectured by
Brualdi and Li \cite{Br} and proved by Drury \cite{Dr}. For $n$ and $d$ fixed, let $l=\lfloor\frac{n}{d}\rfloor$ and $r = n -ld$.
Let $V_j$, $j = 1,\ldots,d$ be disjoint vertex sets, where
$|V_j| = l + 1$ for $j = 1,\ldots,r$ and $|V_j| = l$ for $j = r + 1,\ldots,d$. Let $G_0$ be a digraph with vertex
set $\bigcup\limits_{j=1}^d V_j$ with all possible arcs between $V_j$ and $V_k$ for $j\neq k$ and the induced subdigraph $G_0[V_j]$ an
extremal tournament for $j = 1,\ldots,d$.

It is natural to ask: which digraph achieves the maximum $A_\alpha$ spectral radius among all strongly connected digraphs with given clique number?
If $\alpha=0$, Drury and Lin \cite{DL1} proved that $G_0$ attains maximum $A_0$ spectral radius
among all strongly connected digraphs with given clique number $d$. Therefore, we propose the following problem.

\noindent\begin{problem}\label{pr:c-1} Let $1\leq d\leq n-1$.
Among all digraphs in $\mathcal{C}_{n,d}$, does $G_0$ attains maximum $A_\alpha$ spectral radius?
\end{problem}

\section{The maximum $A_\alpha$ spectral radius of strongly connected digraphs with given vertex connectivity}

Let $\mathcal{D}_{n,k}$ denote the set of strongly connected digraphs
with order $n$ and vertex connectivity $\kappa(G)=k\geq1$. If $k=n-1$,
then $\mathcal{D}_{n,k}=\{\overset{\longleftrightarrow}{K_{n}}\}$.
So we only consider the cases $1\leq k \leq n-2$.

Foe $1\leq m \leq n-k-1$, $K(n,k,m)$ denote the digraph $\overset{\longleftrightarrow}{K_k}\vee(\overset{\longleftrightarrow}{K_n}_{-k-m}\cup
\overset{\longleftrightarrow}{K_m})+ E$, where $E=\{(u,v)|u\in V(\overset{\longleftrightarrow}{K_{m}}),
v\in V(\overset{\longleftrightarrow}{K_n}_{-k-m})\}$ (see Figure \ref{Fig.1.}).
Let $\mathcal{K}(n,k)=\{K(n,k,m): \ 1\leq m \leq n-k-1\}$. Clearly $\mathcal{K}(n,k)\subset\mathcal{D}_{n,k}$.

In \cite{LSWY}, Lin and Shu et al. proved that $K(n,k,n-k-1)$ or $K(n,k,1)$ attains the maximum $A_0$ spectral radius among all strongly connected
digraphs with given vertex connectivity. In \cite{HoYo}, Hong and You determined that $K(n,k,n-k-1)$ also attains the maximum $A_{\frac{1}{2}}$ spectral radius
among all strongly connected digraphs with given vertex connectivity. We generalize their results to $0\leq\alpha<1$.

\begin{figure}[H]
\begin{centering}
\includegraphics[scale=0.6]{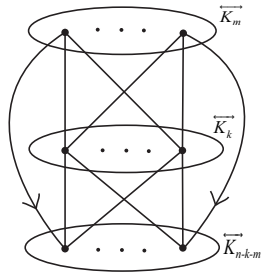}
\caption{The digraph $K(n,k,m)$}\label{Fig.1.}
\end{centering}
\end{figure}


\noindent\begin{lemma}\label{le:11} (\cite{BoMu}) Let $G$ be a strongly
connected digraph with $\kappa(G)=k$. Suppose that $S$ is a $k$-vertex cut
of $G$ and $G_1,G_2,\ldots,G_t$ are the strongly connected components of $G-S$.
Then there exists an ordering of $G_1,G_2,\ldots,G_t$ such that for $1\leq i\leq t$
and any $v\in V(G_i)$, every tail of $v$ is in $\bigcup^i_{j=1}G_{j}$.
\end{lemma}

\noindent\begin{remark}\label{re:1}
By Lemma \ref{le:11}, we know that there exists a strongly
connected component of $G-S$, say $G_1$ with $|V(G_1)|=m$ such that
for any $i\in V(G_1)$, $|N^-_i|=0$, where $N_i^-=\{j\in V(G-S-G_1): (j,i)\in E(G)\}$.
Let $G_2=G-S-G_1$. We add arcs to $G$ until both induced subdigraph
of $V(G_1)\cup S$ and induced subdigraph of $V(G_2)\cup S$ attain to
complete digraphs, add arc $(u,v)$ for any $u\in V(G_1)$ and any $v\in V(G_2)$.
Denote the resulting digraph by $H$. Since $G$ is $k$-strongly connected, then
$H=K(n,k,m)\in\mathcal{K}(n,k)\subset\mathcal{D}_{n,k}$. By Corollary \ref{co:1}, we have
$\lambda_\alpha(G)\leq\lambda_\alpha(H)$, with equality if and only if $G\cong H$. Therefore,
the digraph which achieves the maximum $A_\alpha$ spectral radius among all digraphs
in $\mathcal{D}_{n,k}$ must be some digraph in $\mathcal{K}(n,k)$.
\end{remark}

\noindent\begin{theorem}\label{th:c-5} Let $n,k,m$ be positive integers
such that $1\leq k \leq n-2$ and $1\leq m \leq n-k-1$. Then
$$\resizebox{.9\hsize}{!}
{$\lambda_\alpha(K(n,k,m))= \frac{n-2-\alpha m+\alpha n+\sqrt{
(1-\alpha)^2n^2+(6\alpha-2\alpha^2-4)mn+(2-\alpha)^2m^2+4(1-\alpha)km}}{2}$}.$$
\end{theorem}

\begin{proof} Let $G=K(n,k,m)$, and $S$ be a $k$-vertex cut of
$G$. Suppose that $G_1$ with $|V(G_1)|=m$ and $G_2$ with $|V(G_2)|=n-k-m$
are two strongly connected components, i.e., two complete subdigraphs of $G-S$
with arcs $\{(u,v): u\in V(G_1), v\in V(G_2)\}$. Then
$$\resizebox{.9\hsize}{!}{$A_\alpha(G)=\left(
  \begin{array}{ccc}
     (1-\alpha)J_m+(\alpha n-1)I_m & (1-\alpha)J_{m\times k} & (1-\alpha)J_{m\times (n-k-m)}\\
     (1-\alpha)J_{k\times m}  & (1-\alpha)J_k+(\alpha n-1)I_k & (1-\alpha)J_{k\times (n-k-m)} \\
    0_{(n-k-m)\times m} & (1-\alpha)J_{(n-k-m)\times k} & (1-\alpha)J_{n-k-m}+(\alpha(n-m)-1)I_{n-k-m}\\
  \end{array}
\right)$},$$
where $I_p$ be the $p\times p$ identity matrix and $J_{p\times q}$ be
the $p\times q$ matrix in which every entry is 1, or simply $J_p$ if $p=q$, and $0_{p\times p}$ denotes
the $p\times q$ zero matrix.

Note that the matrix $A_\alpha(G)$ has the following equitable quotient matrix $B(G)$ with respect to the partition
$\{V(G_1),S, V(G_2)\}$ of $V(G)$.
$$B(G)=\left(
  \begin{array}{cccc}
  \alpha(n-m)+m-1 &(1-\alpha)k & (1-\alpha)(n-k-m) \\
    (1-\alpha)m & \alpha(n-k)+k-1 &(1-\alpha)(n-k-m) \\
    0 & (1-\alpha)k& n-k-m-1+\alpha k
  \end{array}
\right).$$
Then by Lemma \ref{le:c1}, $\lambda_\alpha(G)$ is also the spectral radius of $B(G)$. Since $A_\alpha(G)$ is nonnegative
irreducible matrix, $B(G)$ is also nonnegative irreducible matrix. By Corollary \ref{co:c3}, we have
$\rho(B)=\lambda_\alpha(G)>\alpha(n-1)>\alpha n-1$. Therefore, we can easily see that
$\lambda_\alpha(G)$ is the largest root of the quadratic equation
$x^2-(\alpha n+n-\alpha m-2)x+\alpha n^2-\alpha n-2\alpha nm-m^2+\alpha km+\alpha m+\alpha m^2-n+mn+1-km=0$, thus we have
$$\resizebox{.9\hsize}{!}
{$\lambda_\alpha(K(n,k,m))= \frac{n-2-\alpha m+\alpha n+\sqrt{
(1-\alpha)^2n^2+(6\alpha-2\alpha^2-4)mn+(2-\alpha)^2m^2+4(1-\alpha)km}}{2}$}.$$
\end{proof}

\noindent\begin{remark} Noting that $\overset{\longleftrightarrow}{K_{n}}$ is the unique
digraph which achieves the maximum $A_\alpha(G)$ spectral radius $n-1$ among all strongly connected
digraphs, and $K(n,n-2,1)=\overset{\longleftrightarrow}{K_{n}}-\{(u,v)\}$
where $u,v\in V(\overset{\longleftrightarrow}{K_{n}})$, by Lemma \ref{le:2}
and Theorem \ref{th:c-5}, we deduce that $K(n,n-2,1)$ is the unique
digraph which achieves the second maximum $A_\alpha(G)$ spectral radius
$$\frac{n+\alpha n-2-\alpha+\sqrt{
(1-\alpha)^2n^2+2\alpha(1-\alpha)n+\alpha^2+4\alpha-4}}{2}$$ among all
strongly connected digraphs.
\end{remark}

\noindent\begin{theorem}\label{th:c-6} Let $n,k$ be positive integers
such that $1\leq k \leq n-2$, $G\in\mathcal{D}_{n,K}$. Then

$(i)$ for $\alpha=0$, $\lambda_\alpha(G)\leq\frac{n-2+\sqrt{n^2-4n+4k+4}}{2}$ with equality if and only if
$G\cong K(n,k,n-k-1)$ or $G\cong K(n,k,1)$.

$(ii)$ For $0<\alpha<1$, $\lambda_\alpha(G)\leq\frac{n-2+\alpha+\alpha k+\sqrt{n^2+(2\alpha-4-2\alpha k)n+\alpha^2+\alpha^2k^2-4\alpha+2\alpha^2k-4\alpha k+4k+4}}{2},$
with equality if and only if $G\cong K(n,k,n-k-1)$.
\end{theorem}

\begin{proof} By Remark \ref{re:1}, $\lambda_\alpha(G)\leq K(n,k,m)$ for some $m$, where $1\leq m \leq n-k-1$. By Theorem \ref{th:c-5}, we have
$\resizebox{.85\hsize}{!}
{$\lambda_\alpha(K(n,k,m))= \frac{n-2-\alpha m+\alpha n+\sqrt{
(1-\alpha)^2n^2+(6\alpha-2\alpha^2-4)mn+(2-\alpha)^2m^2+4(1-\alpha)km}}{2}$}.$ Now we want to show that the maximum value of $\lambda_\alpha(K(n,k,m))$ must be
taken at either $m=1$ or at $m=n-k-1$.

Let $\resizebox{.9\hsize}{!}
{$ f(n,k,m)=n-2-\alpha m+\alpha n+\sqrt{
(1-\alpha)^2n^2+(6\alpha-2\alpha^2-4)mn+(2-\alpha)^2m^2+4(1-\alpha)km}$}$.
\\Then
$$\frac{\partial f}{\partial m}=-\alpha+\frac{1}{2}\frac{(6\alpha-2\alpha^2-4)n+2(2-\alpha)^2m+4(1-\alpha)k}{\sqrt{
(1-\alpha)^2n^2+(6\alpha-2\alpha^2-4)mn+(2-\alpha)^2m^2+4(1-\alpha)km}},$$
\begin{align*}
\frac{\partial^2 f}{\partial m^2}&=\frac{1}{4}\frac{16k^2(2\alpha-\alpha^2-1)+16nk(-5\alpha+4\alpha^2+2-\alpha^3)}{(
(1-\alpha)^2n^2+(6\alpha-2\alpha^2-4)mn+(2-\alpha)^2m^2+4(1-\alpha)km)^{\frac{3}{2}}}\\
&=\frac{1}{4}\frac{-16k^2(\alpha-1)^2+16nk(2-\alpha)(\alpha-1)^2}{(
(1-\alpha)^2n^2+(6\alpha-2\alpha^2-4)mn+(2-\alpha)^2m^2+4(1-\alpha)km)^{\frac{3}{2}}}\\
&\geq\frac{1}{4}\frac{-16k^2(\alpha-1)^2+16(k+2)k(2-\alpha)(\alpha-1)^2}{(
(1-\alpha)^2n^2+(6\alpha-2\alpha^2-4)mn+(2-\alpha)^2m^2+4(1-\alpha)km)^{\frac{3}{2}}}\\
&=\frac{1}{4}\frac{16k^2(\alpha-1)^2(1-\alpha)+32k(2-\alpha)(\alpha-1)^2}{(
(1-\alpha)^2n^2+(6\alpha-2\alpha^2-4)mn+(2-\alpha)^2m^2+4(1-\alpha)km)^{\frac{3}{2}}}
>0.
\end{align*}
Thus, for fixed $n$ and $k$, the maximum value of $f$ must be
taken at either $m=1$ or at $m=n-k-1$.

In the following, we want to compare $\lambda_\alpha(K(n,k,1))$ and $\lambda_\alpha(K(n,k,n-k-1))$.
Let $\beta=(\alpha^2+1-2\alpha)n^2+(6\alpha-2\alpha^2-4)n+4-4\alpha+\alpha ^2-4\alpha k+4k$
and $\gamma=n^2+(2\alpha-4-2\alpha k)n+\alpha^2+\alpha^2k^2-4\alpha+2\alpha^2k-4\alpha k+4k+4$. Then
\begin{align*}
f(n,k,n-k-1)-f(n,k,1)&=2\alpha+\alpha k-\alpha n+\sqrt{\gamma}-\sqrt{\beta}\\
&=\alpha(n-k-2)(-1+\frac{2n-\alpha(k+n)}{\sqrt{\gamma}+\sqrt{\beta}}).
\end{align*}

For $\alpha=0$, we have $f(n,k,n-k-1)-f(n,k,1)=0$, that is $\lambda_\alpha(n,k,n-k-1)=\lambda_\alpha(n,k,1)=\frac{n-2+\sqrt{n^2-4n+4k+4}}{2}$
 Therefore, $\lambda_\alpha(G)\leq\frac{n-2+\sqrt{n^2-4n+4k+4}}{2}$ with equality if and only if
$G\cong K(n,k,n-k-1)$ or $G\cong K(n,k,1)$.

For $0<\alpha<1$. We assume that $n>k+2$ since in case $n=k+2$ there is only one value of
$m$ under consideration. Now suppose that $f(n,k,n-k-1)-f(n,k,1)\leq0$. We deduce a contradiction. We have simultaneously
$$\sqrt{\gamma}+\sqrt{\beta}\geq 2n-\alpha(k+n) \ \textrm{ and}  \ -\sqrt{\gamma}+\sqrt{\beta}\geq2\alpha+\alpha k-\alpha n.$$
So $\sqrt{\beta}\geq n-\alpha n+\alpha$ and $\beta\geq(n-\alpha n+\alpha)^2$.
However, $(n-\alpha n+\alpha)^2-\beta=-4\alpha n+4n-4+4\alpha k+4\alpha-4k=4(1-\alpha)(n-k-1)>0$, that is $(n-\alpha n+\alpha)^2>\beta$.
Thus $f(n,k,n-k-1)-f(n,k,1)>0$. Therefore,
$$\resizebox{.9\hsize}{!}
{$\lambda_\alpha(G)\leq\lambda_\alpha(n,k,n-k-1)=\frac{n-2+\alpha+\alpha k+\sqrt{n^2+(2\alpha-4-2\alpha k)n+\alpha^2+\alpha^2k^2-4\alpha+2\alpha^2k-4\alpha k+4k+4}}{2}$},$$
with equality if and only if $G\cong K(n,k,n-k-1)$.
\end{proof}

\section{The maximum $A_\alpha$ spectral radius of strongly connected digraphs with given arc connectivity}

Let $\mathcal{L}_{n,k}$ denote the set of strongly connected digraphs
with order $n$ and arc connectivity $\kappa'(G)=k\geq1$. If $\kappa'(G)=k=n-1$,
then $\mathcal{L}_{n,k}=\{\overset{\longleftrightarrow}{K_{n}}\}$.
So we only consider the cases $1\leq k \leq n-2$.

In  \cite{LD}, Lin and Drury proved that $K(n,k,n-k-1)$ or $K(n,k,1)$ attains the maximum $A_0$ spectral radius among all strongly connected
digraphs with given arc connectivity. In \cite{XiWa3}, Xi and Wang determined that $K(n,k,n-k-1)$ also attains the maximum $A_{\frac{1}{2}}$ spectral radius
among all strongly connected digraphs with given arc connectivity. We generalize their results to $0\leq\alpha<1$.

\noindent\begin{lemma}\label{le:12} (\cite{XiWa3}) Let $G$ be a strongly connected digraph
with order $n$ and arc connectivity $k\geq1$, and $S$ be an arc cut set of $G$ of size $k$
such that $G-S$ has exactly two strongly connected components, say $G_1$ and $G_2$ with $|V(G_1)|=n_1$ and $|V(G_2)|=n_2$, where $n_1+n_2=n$.
If $d_v^+>k$ and $d_v^->k$ for each vertex $v\in V(G)$, then $n_1\geq k+2$, $n_2\geq k+2$.
\end{lemma}

\noindent\begin{lemma}\label{le:13} Let $G\in\mathcal{L}_{n,k}$ containing a vertex of outdegree $k$. Then
$\lambda_\alpha(G)\leq \lambda_\alpha(K(n,k,n-k-1))$.
\end{lemma}

\begin{proof} Let $w$ be a vertex of $G$ such that $d_w^+=k$. Then the arcs out-incident to $w$ form an arc cut set of size $k$. Adding all possible arcs from $G\setminus\{w\}$ to $G\setminus\{w\}\cup\{w\}$, we obtain a digraph $H$, which is isomorphic to $K(n,k,n-k-1)$, the arc connectivity of $H$ remains equal to $k$. If
$G\neq K(n,k,n-k-1)$, then $\lambda_\alpha(G)<\lambda_\alpha(K(n,k,n-k-1))$ by Corollary \ref{co:1}. So the result follows.
\end{proof}

\noindent\begin{lemma}\label{le:14} Let $G\in\mathcal{L}_{n,k}$ containing a vertex of indegree $k$. Then
$\lambda_\alpha(G)\leq \lambda_\alpha(K(n,k,1))$.
\end{lemma}

\begin{proof} Let $w$ be a vertex of $G$ such that $d_w^-=k$. Then the arcs in-incident to $w$ form an arc cut set of size $k$. Adding all possible arcs from $w$ to $G\setminus\{w\}$, and all possible arcs from $G\setminus\{w\}$ to $G\setminus\{w\}$, we obtain a digraph $H'$, which is isomorphic to $K(n,k,1)$, the arc connectivity of $H'$ remains equal to $k$. If
$G\neq K(n,k,1)$, then $\lambda_\alpha(G)<\lambda_\alpha(K(n,k,1))$ by Corollary \ref{co:1}. So the result follows.
\end{proof}

Let $\delta^0(G)=\min\{\delta^+(G), \delta^-(G)\}$. The following result characterize the digraphs maximizes the
$A_\alpha$ spectral radius in $\mathcal{L}_{n,k}$ when $\kappa'(G)=\delta^0(G)=k\geq1$.

\noindent\begin{theorem}\label{th:c-7} Let $G\in\mathcal{L}_{n,k}$ with $\delta^0(G)=k$. Then
(i) if $\alpha=0$, $\lambda_\alpha(G)\leq\lambda_\alpha(K(n,k,1))=\lambda_\alpha(K(n,k,n-k-1))$ with equality if and only if
$G\cong K(n,k,n-k-1)$ or $G\cong K(n,k,1)$.

(ii) If $0<\alpha<1$, $\lambda_\alpha(G)\leq\lambda_\alpha(K(n,k,n-k-1)),$
with equality if and only if $G\cong K(n,k,n-k-1)$.
\end{theorem}

\begin{proof} Let $G$ be a digraph with arc connectivity $k$, and $\delta^0(G)=k$.

If $\delta^0(G)=\delta^+(G)$, then there exists a vertex of outdegree $k$. So by Lemma \ref{le:13},
$K(n,k,n-k-1)$ maximizes the $A_\alpha$ spectral radius in $\mathcal{L}_{n,k}$.
If $\delta^0(G)=\delta^-(G)$, then there exists a vertex of indegree $k$. So by Lemma \ref{le:14},
$K(n,k,1)$ maximizes the $A_\alpha$ spectral radius in $\mathcal{L}_{n,k}$.  Moreover, by the proof Theorem \ref{th:c-6}, we know that

$(i)$ For $\alpha=0$, $\lambda_\alpha(K(n,k,1))=\lambda_\alpha(K(n,k,n-k-1))$.
Therefore $\lambda_\alpha(G)\leq\lambda_\alpha(K(n,k,1))=\lambda_\alpha(K(n,k,n-k-1))$ with equality if and only if
$G\cong K(n,k,n-k-1)$ or $G\cong K(n,k,1)$.

$(ii)$ For $0<\alpha<1$, $\lambda_\alpha(K(n,k,n-k-1))>\lambda_\alpha(K(n,k,1))$. Therefore $\lambda_\alpha(G)\leq\lambda_\alpha(K(n,k,$ $n-k-1))$
with equality if and only if $G\cong K(n,k,n-k-1)$.
\end{proof}

\noindent\begin{theorem}\label{th:c-8}
Let $G\in\mathcal{L}_{n,k}$. Then

$(i)$ for $\alpha=0$, $\lambda_\alpha(G)\leq\lambda_\alpha(K(n,k,1))=\lambda_\alpha(K(n,k,n-k-1))$ with equality if and only if
$G\cong K(n,k,n-k-1)$ or $G\cong K(n,k,1)$.

$(ii)$ For $0<\alpha<1$, $\lambda_\alpha(G)\leq\lambda_\alpha(K(n,k,n-k-1))$, with equality if and only if $G\cong K(n,k,n-k-1)$.
\end{theorem}

\begin{proof} Let $G$ be a digraph in $\mathcal{L}_{n,k}$. Note that each vertex in the digraph $G$ has outdegree at least
$k$ and indegree at least $k$, otherwise $G\notin\mathcal{L}_{n,k}$. Then, we consider the following two cases.

{\bf Case 1.} If there exists a vertex $u$ of $G$ with $d_u^+=k$ or $d_u^-=k$. The by Theorem \ref{th:c-7}, we get the result.

{\bf Case 2.}  For arbitrary vertex $u$ of $G$, $d_u^+>k$ and $d_u^->k$. Let $S$ be an arc cut set of $G$ containing $k$ arcs, then $G-S$ consists of exactly two strongly
connected components $G_1$, $G_2$, respectively, of orders $a$, $b$ and $a+b=n$. Without loss of generality, we may assume that there are no arcs from
$G_1$ to $G_2$ in $G-S$. By Lemma \ref{le:12}, $a\geq k+2$, $b=n-a\geq k+2$, then $k+2\leq a\leq n-k-2$, $n\geq a+k+2\geq2k+4$.
Next we create a new digraph $\overline{H}$ by adding to $G$ any possible arcs from $G_2$
to $G_1\cup G_2$ or any possible arcs from $G_1$ to $G_1$ that were not present in $G$. Obviously,
the arc connectivity of $\overline{H}$ remains equal to $k$ and all
vertices of $\overline{H}$ have outdegree greater than $k$ and indegree still greater than $k$. By Corollary \ref{co:1},
the addition of any such arc will give $\lambda_\alpha(\overline{H})>\lambda_\alpha(G)$. For $k+2\leq a\leq n-k-2$, let $H''=\overset{\longleftrightarrow}{K_{a}}\cup\overset{\longleftrightarrow}{K_{n}}_{-a}$, $U=\{u_1, u_2, \cdots, u_k\}$ be a set of $k$ vertices in
$V(\overset{\longleftrightarrow}{K_{a}})$ and $W=\{v_1, v_2, \cdots, v_k\}$ be a set of $k$ vertices in $V(\overset{\longleftrightarrow}{K_n}_{-a})$. Let
$H_4$ be a digraph obtained from $H''$ by adding all possible arcs from $U$ to $W$, and adding all possible arcs from $\overset{\longleftrightarrow}{K_{n}}_{-a}$ to $\overset{\longleftrightarrow}{K_{a}}$. Then $\overline{H}$ be a spanning subdigraph of $H_4$.
Therefore, by Corollary \ref{co:1}, $\lambda_\alpha(\overline{H})\leq \lambda_\alpha(H_4)$. However, we can know that the vertex connectivity of $H_4$ is $k$ and
$H_4\ncong K(n,k,n-k-1)$ and $H_4\ncong K(n,k,1)$. Then by Theorem \ref{th:c-6}, if $\alpha=0$,
we have $\lambda_\alpha(H_4)<\lambda_\alpha(K(n,k,n-k-1))=\lambda_\alpha(K(n,k,1))$; if $0<\alpha<1$, $\lambda_\alpha(H_4)<\lambda_\alpha(K(n,k,n-k-1))$.

Therefore, combining the above two cases, we get the desired result.
\end{proof}

In sections 6 and 7, it is natural to ask: what are digraphs in $\mathcal{D}_{n,k}$ and $\mathcal{L}_{n,k}$
whose $A_\alpha$ spectral radius is minimum for each $1\leq k\leq n-2$, respectively?

By Lemma \ref{le:1}, for any strongly connected digraph $H$, $\delta^+\leq \lambda_\alpha(H)\leq\Delta^+$, with
either equality if and only if the outdegrees of all vertices in $H$ are equal. If  $G\in\mathcal{D}_{n,k}$
or $G\in\mathcal{L}_{n,k}$, then $\delta^+\geq k$, $\delta^-\geq k$. So we have $\lambda_\alpha(G)\geq k$
with the equality if and only if the outdegrees of all vertices in $G$ is $k$, and the indegree of all vertices in $G$ is also $k$ because $\delta^-\geq k$. That is $\lambda_\alpha(G)\geq k$ with the equality if and only if $G$ is a $k$-regular digraph.

\noindent\begin{theorem}\label{th:ch-9} For each $1\leq k\leq n-2$, a digraph with the minimum $A_\alpha$ spectral radius in
$\mathcal{D}_{n,k}$ or $\mathcal{L}_{n,k}$ is $k$-regular, and
$$\min\{\lambda_\alpha(G) : G\in\mathcal{D}_{n,k}\}=\min\{\lambda_\alpha(G) : G\in\mathcal{L}_{n,k}\}=k.$$
\end{theorem}

\end{document}